\documentclass[a4paper, 12pt]{article}
\usepackage{amsmath, amsfonts, amssymb, amsthm}
\usepackage[english]{babel}

\newcounter{num}[section]

\newenvironment{theorem}
{\refstepcounter{num}%
\bigskip\noindent\nopagebreak[4]{\bf Theorem~\arabic{section}.\arabic{num}. }\it}

\newenvironment{lemma}
{\refstepcounter{num}%
\bigskip\noindent\nopagebreak[4]{\bf Lemma~\arabic{section}.\arabic{num}. }\it}

\newenvironment{remark}
{\refstepcounter{num}%
\bigskip\noindent\nopagebreak[4]{\bf Remark~\arabic{section}.\arabic{num}. }}

\newcommand{\LL}{{\mathcal{L}}}

\newcommand{\Ss}{{\mathcal{S}}}
\newcommand{\V}{{\mathrm{V}}}

\newcommand{\pr}{{\prime}}
\renewcommand{\t}{{\tau}}
\newcommand{\s}{{\sigma}}

\newcommand{\M}{{\mathcal{M}}}

\newcommand{\dom}{{\mathbf{dom}}}
\newcommand{\im}{{\mathbf{im}}}

\begin{document}

\author{Artem N. Shevlyakov}

\title{On disjunction of equations in inverse semigroups}

\maketitle

\abstract{A semigroup $S$ is an equational domain if any finite union of algebraic sets over $S$ is algebraic. We prove that if an inverse semigroup $S$ is an equational domain in the extended language $\{\cdot,{}^{-1}\}\cup\{s|s\in S\}$ then $S$ is a group.}

\section*{Introduction}

It follows from commutative algebra that the union of two algebraic sets $Y_1,Y_2$ is algebraic over a field $k$. There are also groups, where any finite union of algebraic sets is again algebraic. Following~\cite{uniTh_IV}, such groups are called  {\it equational domains}, and in~\cite{uniTh_IV} the equational domains among groups were completely described. 

The complete classification of equational domains in groups allows us to pose the next problem.

{\bf Problem.} {\it Is there a semigroup $S$ such that
\begin{enumerate}
\item $S$ is an equational domain;
\item $S$ is not a group.
\end{enumerate}}

Naturally, the search of such semigroups should be started with the classes of semigroups which are close to groups. One of these classes is the variety of inverse semigroups with the operations of multiplication and inversion. However, we prove below that any inverse semigroup (which is not a group) is not an equational domain (Theorem~\ref{th:main_for_inverse}). 

Remark that the similar problem was considered in~\cite{rosenblatt}. Namely, there it was proved that the union $Y$ of the solution sets of two equations $x_1=x_2$, $x_3=x_4$ is not a solution set of any single equation over any inverse semigroup which is not a group. However, the approach in~\cite{rosenblatt} does not allow to prove that $Y$ is not algebraic, i.e. $Y$ is not a solution set of a {\it system of equations}. Thus, our paper generalizes the results of~\cite{rosenblatt} (see Remark~\ref{rem:about_rosenblat} for explanation of the method in~\cite{rosenblatt}).

\section{Definitions of semigroup theory}

Let us give the main definitions of semigroup theory which, for more details one can recommend~\cite{howie}.

A semigroup $S$ is called \textit{inverse} if for any $s\in S$ there exists a unique element $s^{-1}$ such that $ss^{-1}s=s$, $s^{-1}ss^{-1}=s^{-1}$. All necessary properties of inverse semigroups are given in the next theorem.

\begin{theorem}
\label{th:inverse_semigroups_properties}
Let $E=\{e\in S|ee=e\}$ be the set of idempotents of an inverse semigroup $S$. Hence,
\begin{enumerate}
\item all elements of $E$ commute with each other; in other words, $E$ is a semilattice;
\item $ss^{-1}\in E$;
\item for any $e\in E$, $s\in S$ it holds $ses^{-1}\in E$;
\item if $|E|=1$, $S$ if a group;
\item $(st)^{-1}=t^{-1}s^{-1}$ for all $s,t\in S$.
\end{enumerate}
\end{theorem}

Let us define the partial order over the set of idempotents by 
\[
e\leq f \Leftrightarrow ef=e.
\]

Let $\M$ be an arbitrary set and $T(\M)$ be the class of all partial injections $f\colon \M\to\M$. 

Define for $f\in T(\M)$ 
\begin{eqnarray*}
\dom(f)=\{\mu\in \M|f\mbox{ is defined at the point } \mu\},\\
\im(f)=\{\mu\in \M|\exists \eta\in\M |\mbox{such that } \mu f=\eta\}.
\end{eqnarray*}

\begin{lemma}
Partial injections have the following properties:
\begin{enumerate}
\item if $f\in T(\M)$ is idempotent (i.e. $ff=f$), then $\dom(f)=\im(f)$ and $\mu f=\mu$ for all $\mu\in\dom(f)$; 
\item for any $f\in T(\M)$ the map $ff^{-1}$ is idempotent and $\dom(ff^{-1})=\dom(f)$;
\item for idempotent $e,f\in T(\M)$ it holds $\dom(ef)=\dom(e)\cap\dom(f)$;
\item the partial order  $\leq$ over the set of idempotents in $T(\M)$ is  
\[
e\leq f\Leftrightarrow \dom(e)\subseteq\dom(f).
\]
\end{enumerate}
\end{lemma} 

By $f|_g$ ($f,g\in T(\M)$) we denote the restriction of $f$ on the set $\dom(g)$.

The following theorem gives the connections between inverse semigroups and partial injections.

\begin{theorem}\textup{(Wagner, Preston)}
\label{th:wagner-preston}
Any inverse semigroup $S$ is embedded into $T(\M)$ for an appropriate set $\M$.
\end{theorem}

\section{Algebraic geometry over inverse semigroups}

Let $\LL_0=\{\cdot\}$ be the standard language of semigroup theory. For a given semigroup $S$ one can extend the language $\LL_0$ by the constants corresponding to the elements of $S$
\[
\LL_S=\LL_0\cup\{s|s\in S\}.
\]

Let $S$ be an inverse semigroup. As the operation ${ }^{-1}$ is uniquely defined for any $s\in S$, one can add the inversion to the language $\LL_S$ and obtain
\[
\LL=\{\cdot,{ }^{-1}\}\cup\{s|s\in S\}.
\]

Further, any semigroup in the language $\LL$ is called an $\LL$-{\it semigroup}. 

Denote by $X$ the finite set of variables $x_1,x_2,\ldots,x_n$. A \textit{term} of the language $\LL$ in variables $X$ is a finite product of variables in integer degrees and constants.

\begin{remark}
Precisely, a term should be defined recursively by follows:
\begin{enumerate}
\item any variable or constant is a term;
\item the product of two terms is a term;
\item if $t(X)$ is a term, so is $(t(X))^{-1}$. 
\end{enumerate}
However, by the identity $(xy)^{-1}=y^{-1}x^{-1}$, one can assume that the inversion is applied to single variables. For example, the term $((s_1x)^{-1}ys_2)^{-1}$ is equivalent to
\[
s_2^{-1}y^{-1}s_1x=s_3y^{-1}s_1x,
\]
where $s_3=s_2^{-1}$.
Thus, the both definitions of a term over an inverse semigroup are equivalent. 
\end{remark}

\medskip

An {\it equation} over the language $\LL$ is an equality of two terms $\t(X)=\s(X)$. For example, the expressions  $x_1^{-1}sx_2^2=x_3^{-2}sx_1$, $s_1x_1^2x_2^{-1}s_2x_3=x_1^{-5}x_3$ are equations over $\LL$. A {\it system of equations} (a {\it system} for shortness) is an arbitrary set of equations.
 
The \textit{solution set} of a system $\Ss$ in the $\LL$-semigroup $S$ is naturally defined and denoted by $\V_S(\Ss)$. A set $Y\subseteq S^n$ is {\it algebraic} over a semigroup $S$ if there exists a system $\Ss$ in variables $x_1,x_2,\ldots,x_n$ with the solution set $Y$. 

Following~\cite{uniTh_IV}, let us give the main definition of the paper. An $\LL$-semigroup $S$ is called an  {\it equational domain (e.d.)} if any finite union $Y=Y_1\cup Y_2\cup\ldots\cup Y_n$ of algebraic sets $Y_i$ is algebraic.

Following~\cite{rosenblatt}, a term $t(x)$ of the language $\LL$ is called \textit{good} if one of the next conditions holds:
\begin{enumerate}
\item $t(x)=x$;
\item there exists constants $s_1,s_2,\ldots,s_n\in S$ with
\[
t(x)=s_1xs_1^{-1}s_2xs_2^{-1}\ldots s_nxs_n^{-1}.
\]
\end{enumerate}

The next lemma gives us the main properties of good terms.

\begin{lemma}\textup{\cite{rosenblatt}}
\label{l:good_terms}
Let $E$ be the set of idempotents of an inverse semigroup $S$. Hence,
\begin{enumerate}
\item if $e\in E$, then $t(e)\in E$ for any good term $t(x)$;
\item if $e,f\in E$ and $t(x)$ is good, then $t(ef)=t(e)t(f)$;
\item if $e,f\in E$, $e\leq f$, then for any good term $t(x)$ it holds $t(e)\leq t(f)$;
\item for any term $t(x)$ of the language $\LL$ there exists a good term $t^\pr(x)$ such that $t(e)=t^\pr(e)$ for any $e\in E$;
\item for any term $t(x,y)$ of the language $\LL$ there exists good terms $t^\pr(x),r^\pr(y)$ and a element $d\in S$ such that $t(e,f)=t^\pr(e)r^\pr(f)d$ for all $e,f\in E$.
\end{enumerate}
\end{lemma}

\section{Main result}

\begin{lemma}
\label{l:incompatible_idempotents}
Suppose an inverse semigroup $S$ contains two incomparable idempotents $e,f$, therefore $S$ is not an e.d. in the language $\LL$. 
\end{lemma}
\begin{proof}

Assume there exists a system $\Ss(x)$ with the solution set 
\begin{equation}
\label{eq:S=x=e_vee_x=f}
\V_S(\Ss)=\V_S(x=e)\cup\V_S(x=f).
\end{equation} 
Below we shall prove that the point $ef$ satisfies $\Ss$, and obtain the contradiction.

Let $w(x)=w^\pr(x)$ be an equation of $\Ss$. By Lemma~\ref{l:good_terms}, there exists good terms $t(x), t^\pr(x)$ and elements $d,d^\pr\in S$ such that the equation $w(x)=w^\pr(x)$ is equivalent to $t(x)d=t^\pr(x)d^\pr$ over the set $E$.

From the equality~(\ref{eq:S=x=e_vee_x=f}) we have
\begin{eqnarray}
\label{eq:ololo1}
t(e)d=t^\pr(e)d^\pr,\\
\label{eq:ololo2}
t(f)d=t^\pr(f)d^\pr.
\end{eqnarray}

By Theorem~\ref{th:wagner-preston}, the inverse semigroup $S$ is a subsemigroup in $T(\M)$, hence all elements of $S$ can be considered as partial injections over the set $\M$. Let us show $w(ef)=w^\prime(ef)$, or equivalently
\begin{equation}
t(e)t(f)d=t^\pr(e)t^\pr(f)d^\pr.
\label{eq:ololo3}
\end{equation}

The equality~(\ref{eq:ololo3}) may fail in two cases.

Firstly, the set $\dom(t(e)t(f)d)$ maybe empty. Hence the element $t(e)t(f)d$ is the zero of the semigroup $S$. However, in paper~\cite{ED_I} we proved that any semigroup with a zero is not an e.d.

Secondly, it can be exist an element $\mu\in \dom(t(e)t(f)d)$. Here we have $\mu\in \dom(t(e)d)$, $\mu\in \dom(t(f)d)$. 

By the equalities~(\ref{eq:ololo1},\ref{eq:ololo2}) one can obtain 
$\mu d|_{t(e)}=\mu d^\pr|_{t^\pr(e)}$, $\mu d|_{t(f)}=\mu d^\pr|_{t^\pr(f)}$. Therefore, $\mu d|_{t(e)t(f)}=\mu d^\pr|_{t^\pr(e)t^\pr(f)}$, and it implies the equality~(\ref{eq:ololo3}).
Thus, $ef\in\V_S(\Ss)$ and we came to the contradiction with the choice of the system $\Ss$.
\end{proof}

\begin{lemma}
\label{l:idempotents_linearly_ordered}
Suppose the semilattice of idempotents $E$ of an inverse semigroup $S$ is linearly ordered, and $|E|>1$. Thus, $S$ is not an e.d. in the language $\LL$.
\end{lemma}
\begin{proof}
Assume there exists a system $\Ss(x,y)$ with the solution set
\begin{equation}
\label{eq:S=x=e_vee_y=e}
\V_S(\Ss)=\V_S(x=e)\cup\V_S(y=e),
\end{equation} 
where $e$ is not a minimal idempotent. Below we shall prove that $\Ss$ satisfies $(f,f)$, where $f$ is an arbitrary idempotent less than $e$.

Let $w(x,y)=w^\pr(x,y)\in\Ss$ be an equation which does not satisfy the point $(f,f)$. By Lemma~\ref{l:good_terms}, there exists good terms $t(x), t^\pr(x), r(y), r^\pr(y)$ and elements $d,d^\pr\in S$ such that the equation $w(x,y)=w^\pr(x,y)$ is equivalent to $t(x)r(y)d=t^\pr(x)r^\pr(y)d^\pr$ over the semilattice $E$.

By~(\ref{eq:S=x=e_vee_y=e}), we have
\begin{eqnarray}
\label{eq:ee}
t(e)r(e)d=t^\pr(e)r^\pr(e)d^\pr,\\
\label{eq:ef}
t(e)r(f)d=t^\pr(e)r^\pr(f)d^\pr,\\
\label{eq:fe}
t(f)r(e)d=t^\pr(f)r^\pr(e)d^\pr.
\end{eqnarray}

By the choice of the equation,
\begin{equation}
t(f)r(f)d\neq t^\pr(f)r^\pr(f)d^\pr.
\label{eq:t(f)r(f)d_neq_t(f)r(f)d}
\end{equation}
As all idempotents are linearly ordered, one can put 
\begin{equation}
\label{eq:t(f)r(f)<t(f)r(f)}
t(f)r(f)<t^\pr(f)r^\pr(f).
\end{equation} 

By Theorem~\ref{th:wagner-preston} $S$ is a subsemigroup in $T(\M)$, hence all elements of $S$ are partial injections over the set $\M$. 

There are exactly two cases:
\begin{enumerate}
\item there exists an element $\mu\in\M$ such that $\mu\in\dom(t(f)r(f))$, but $\mu d\neq \mu d^\pr$;
\item there exists $\mu\in\M$ with $\mu\notin\dom(d|_{t(f)r(f)})$, $\mu\in\dom(d^\pr|_{t\pr(f)r^\pr(f)})$.
\end{enumerate} 

Consider the first case. By Lemma~\ref{l:good_terms} we have:
\begin{eqnarray*}
t(f)r(f)\leq t(e)r(e),\\
t^\pr(f)r^\pr(f)\leq t^\pr(e)r^\pr(e),
\end{eqnarray*}
that implies the inclusions:
\begin{eqnarray*}
\dom(t(f)r(f))\subseteq \dom(t(e)r(e)),\\
\dom(t^\pr(f)r^\pr(f))\subseteq\dom(t^\pr(e)r^\pr(e)),
\end{eqnarray*}
Thus, $\mu\in\dom(t(e)r(e))\cap\dom(t^\pr(e)r^\pr(e))$. As $\mu d\neq \mu d^\pr$, we obtain $d|_{t(e)r(e)}\neq d^\pr|_{t^\pr(e)r^\pr(e)}$ that contradicts with the condition~(\ref{eq:ee}).

\medskip

Consider the second case. As all idempotents are linearly ordered, we have exactly four possibilities:
\begin{enumerate}
\item $t(f)\leq r(f)$, $t^\pr(f)\leq r^\pr(f)$. Therefore, $t(f)r(e)=t(f)$, $t^\pr(f)r^\pr(e)=t^\pr(f)$ (here we use $r(f)\leq r(e)$, $r^\pr(f)\leq r^\pr(e)$). Thus, the equality~(\ref{eq:fe}) is reduced to $t(f)d=t^\pr(f)d^\pr$. On the other hand, the inequality~(\ref{eq:t(f)r(f)d_neq_t(f)r(f)d}) becomes $t(f)d\neq t^\pr(f)d^\pr$, and we come to the contradiction.

\item $t(f)\leq r(f)$, $r^\pr(f)\leq t^\pr(f)$. 

The expression~(\ref{eq:t(f)r(f)<t(f)r(f)},\ref{eq:fe}) are reduced to  
\begin{eqnarray*}
t(f)<r^\pr(f),\\
t(f)d=t^\pr(f)r^\pr(e)d^\pr.
\end{eqnarray*}

Hence, $\mu\in \dom(r^\pr(f))\setminus \dom(t(f))$. As $r^\pr(f)\leq t^\pr(f)$, $r^\pr(f)\leq r^\pr(e)$, then $\mu\in\dom(t^\pr(f)r^\pr(e))$. Since $\mu\in\dom(d^\pr)$, we have $\mu\in\dom(t^\pr(f)r^\pr(e)d^\pr)$. Thus, the equality $t(f)d=t^\pr(f)r^\pr(e)d^\pr$ implies $\mu\in\dom(t(f)d)$ that contradicts with the condition $\mu\notin\dom(t(f))$.

\item $r(f)\leq t(f)$, $t^\pr(f)\leq r^\pr(f)$. Therefore, the expressions выражения~(\ref{eq:t(f)r(f)<t(f)r(f)},\ref{eq:ef}) become
\begin{eqnarray*}
r(f)<t^\pr(f),\\
r(f)d=t^\pr(e)r^\pr(f)d^\pr.
\end{eqnarray*}
By the condition, there exists an element $\mu\in\dom(t^\pr(f))\setminus\dom(r(f))$, $\mu\in\dom(d^\pr)$.

As $t^\pr(f)\leq t^\pr(e)$ and $t^\pr(f)\leq r^\pr(f)$, we have $\mu\in\dom(t^\pr(e)r^\pr(f))$. Thus, the equality $r(f)d=t^\pr(e)r^\pr(f)d^\pr$ implies $\mu\in\dom(r(f)d)$ that contradicts with the condition $\mu\notin\dom(r(f))$.

\item $r(f)\leq t(f)$, $r^\pr(f)\leq t^\pr(f)$. In this case the expressions~(\ref{eq:t(f)r(f)<t(f)r(f)},\ref{eq:ef}) become
\begin{eqnarray*}
r(f)<r^\pr(f),\\
r(f)d=r^\pr(f)d^\pr,
\end{eqnarray*}
and there exists an element $\mu\in\dom(r^\pr(f))\setminus\dom(r(f))$, $\mu\in\dom(d^\pr)$. It follows from~(\ref{eq:t(f)r(f)<t(f)r(f)}) that $\mu\in\dom(r(f)d)$, and we come to the contradiction with $\mu\notin\dom(r(f))$. 

\end{enumerate} 

\end{proof}

The lemmas proven above give the main result.

\begin{theorem}
\label{th:main_for_inverse}
If an inverse semigroup $S$ is an e.d. in the language $\LL$, then $S$ is a group.
\end{theorem}
\begin{proof}
By Theorem~\ref{th:inverse_semigroups_properties}, the semigroup $S$ contains at least two idempotents $e,f$. If $e,f$ are incomparable, by Lemma~\ref{l:incompatible_idempotents}, $S$ is not an e.d. Otherwise, one can apply Lemma~\ref{l:idempotents_linearly_ordered}, and $S$ is not an e.d.
\end{proof}

\begin{remark}
\label{rem:about_rosenblat}
In~\cite{rosenblatt} it was proved that the set $\M=\{(x_1,x_2,x_3,x_4)|x_1=x_2\mbox{ or }x_3=x_4\}$ is not a solution set of any equation over an inverse semigroup $S$ (where $S$ is not a group). Proving this result it was assumed that $\M$ equals to the solution set of some equation $t(x_1,x_2,x_3,x_4)=s(x_1,x_2,x_3,x_4)$ of the language $\LL$. Further, the parameters of the equation defined a point $P$ such that 
\[
P\in\V_S(t(x_1,x_2,x_3,x_4)=s(x_1,x_2,x_3,x_4))\setminus\M,
\]
and, finally, it was obtained $\V_S(t(x_1,x_2,x_3,x_4)=s(x_1,x_2,x_3,x_4))\neq\M$.

However, in the current paper we prove that $\M$ is not a solution set of any {\it system of equations}. It is more general statement, and the approach of the paper~\cite{rosenblatt} can not be applied here. 

Indeed, if assume $\M=\V_S(\Ss)$, the existence of a point $P$ defined by some equation $t(x_1,x_2,x_3,x_4)=s(x_1,x_2,x_3,x_4)\in\Ss$ does not lead us to the contradiction, since $P$ is not a solution of the \textit{whole} system $\Ss$.

\end{remark}

The information of the author:

Artem N. Shevlyakov

Omsk Branch of Institute of Mathematics, Siberian Branch of the Russian Academy of Sciences

644099 Russia, Omsk, Pevtsova st. 13

Phone: +7-3812-23-25-51.

e-mail: \texttt{a\_shevl@mail.ru}

\end{document}